\documentclass [article] {amsart}
\usepackage{amsfonts,amsthm,amsbsy,amsmath,amssymb}
\usepackage{latexsym,enumerate,enumitem}
\usepackage{tikz-cd}
\usetikzlibrary{positioning}
\usepackage{xcolor}
\usepackage{mathrsfs}

\newtheorem{lemma}{Lemma}
\newtheorem{prop}{Proposition}
\newtheorem{theorem}{Theorem}

\newtheorem*{theorem*}{Theorem}
\newtheorem*{lemma*}{Lemma}
\newtheorem*{prop*}{Proposition}
\newtheorem*{corollary*}{Corollary}
\newtheorem*{remark*}{Remark} 
\newtheorem*{remarks*}{Remarks}

\def\R{{\mathbb R}}

\def\s{\sigma}

\def\M{\mathcal{M}}

\DeclareFontFamily{U}{mathx}{\hyphenchar\font45}
\DeclareFontShape{U}{mathx}{m}{n}{
	<5> <6> <7> <8> <9> <10>
	<10.95> <12> <14.4> <17.28> <20.74> <24.88>
	mathx10
}{}
\DeclareSymbolFont{mathx}{U}{mathx}{m}{n}
\DeclareFontSubstitution{U}{mathx}{m}{n}
\DeclareMathAccent{\widecheck}{0}{mathx}{"71}
\DeclareMathAccent{\wideparen}{0}{mathx}{"75}

\newcommand{\ry}{r{_{_{Y}} }}
\newcommand{\ryk}{r{_{_{Y_k}} }}

\begin{document} 
 \title[Multilinear Spherical Maximal Function]{Multilinear Spherical Maximal Function}
\author[Georgios Dosidis]{Georgios Dosidis}
\newcommand{\Addresses}{{
		\bigskip
		\footnotesize
		
		\textsc{Department of Mathematics, University of Missouri,
			Columbia MO 65203}\par\nopagebreak
		\textit{E-mail address:} \texttt{ganq8f@mail.missouri.edu}
		

}}
\maketitle
\begin{abstract} In dimensions $n\ge 2$ we obtain $L^{p_1}(\mathbb R^n) \times\dots\times L^{p_m}(\mathbb R^n)$ to $L^p(\mathbb R^n)$ boundedness for the multilinear spherical maximal function in the largest possible open set of indices  and we provide   counterexamples  that  indicate the optimality of our results. Moreover, we obtain  weak type   and Lorentz space estimates  as well as  counterexamples in the endpoint cases.

\end{abstract}
\section*{Introduction}
In this work we  study the $L^p$ boundedness of the $m$-linear analogue of the  spherical maximal function
\begin{equation}\label{MS}
\M (f_1,\dots,f_m) (x) := \sup_{t>0}\left| \int_{\mathbb S^{mn-1}}  \prod_{j=1}^{m} f_j(x-t y^j) d\s_{mn-1}(y^1,\dots,y^m)\right|,
\end{equation}
defined originally for Schwartz functions, where $d\s$ stands for the surface measure of $\mathbb S^{mn-1}$. The study of spherical means was initiated by Stein   \cite{S1976}, who obtained a bound for the linear spherical maximal function 
\begin{equation}\label{SS}
S (f) (x) := \sup_{t>0}\left| \int_{\mathbb S^{n-1}}f(x-t y) d\s_{n-1}(y)\right|
\end{equation}
from $L^p(\mathbb R^n) \to L^p(\mathbb R^n)$ when $n\geq 3$ and $p>\frac{n}{n-1}$ and showed that it is unbounded when $p\leq\frac{n}{n-1}$ and $n\geq 2$. The analogue of this result  in dimension $n=2$ was established by Bourgain in \cite{B1986}, who also obtained a restricted weak type estimate in \cite{B1985} in the case $n\geq 3$. Later, Seeger, Tao, and Wright in \cite{STW2003} proved that the restricted weak type estimate does not hold in dimension $n=2$. A number of other authors have also studied the spherical maximal function; see for instance \cite{C1985}, \cite{CM1979}, \cite{MSS1992}, and \cite{S1998}. Extensions of the spherical maximal function to different settings have also been established by several authors; for instance see \cite{C1979}, \cite{G1981}, \cite{DV1996}, and \cite{MSW2002}. 

The bi(sub)linear analogue of Stein's spherical maximal function was first introduced in by Geba, Greenleaf, Iosevich, Palsson, and Sawyer \cite{GGIPS2013} who obtained the first bounds for it but later improved bounds were provided  by \cite{BGHHO2018}, \cite{GHH2018}, \cite{HHY2019}, and \cite{JL2019}. A multilinear (non-maximal) version of this operator when all input functions lie in the same space $L^{p}(\R)$ was previously studied by Oberlin \cite{O1988}. The authors in \cite{BGHHO2018} provided an example that shows that the bilinear spherical maximal function is not bounded when $p\geq \frac{n}{2n-1}$. Earlier this year, Jeong and Lee in \cite{JL2019} proved that the bilinear maximal function is pointwise bounded by the product of the linear spherical maximal function and the Hardy-Littlewood maximal function, which helped them establish boundedness in the optimal open set of exponents, along with some endpoint estimates. Recently certain analogous bounds have been obtained by Anderson and Palsson   \cite{AP20191}   \cite{AP20192} concerning the discrete multilinear spherical maximal function. 

In this work we extend the results of Jeong and Lee in the multilinear setting and we adapt the counterexample of Barrionuevo, Grafakos, He, Honz\'ik, and Oliveira in \cite{BGHHO2018} to show that our results are sharp. We also provide a counterexample that addresses a question raised by Jeong and Lee in \cite{JL2019} regarding the validity of a strong type $L^1\times L^\infty \to L^1$ bound for the bilinear spherical maximal function.

\section*{Main Results}
Let   $n\geq 2$, $1\leq p_1,\dots ,p_m  \leq \infty$, and $\sum\limits_{j=1}^m \frac{1}{p_j}=\frac{1}{p}$. For the largest possible open set of exponents we prove strong type bounds  
\begin{equation}\label{ST} 
\|\M(f_1,\dots,f_m)\|_{L^p}\lesssim \prod_{j=1}^m \|f_j\|_{L^{p_j}}
\end{equation}
and at the endpoints of this open region we prove weak type estimates 
\begin{equation}\label{WT} \|\M(f_1,\dots,f_m)\|_{L^{p,\infty}}\lesssim \prod_{j=1}^m \|f_j\|_{L^{p_j}}
\end{equation}
 or Lorentz  space  estimates (when $n\geq 3$) of the form 
 \begin{equation}\label{RWT} \|\M(f_1,\dots,f_m)\|_{L^{p,\infty}}\lesssim \Big( \prod_{j\neq k} \|f_j\|_{L^{p_j}}\Big)\|f_k\|_{L^{p_k,1}}, \quad k=1,\dots,m.
 \end{equation}
 Here, as well as in what follows, we write   $A\lesssim B$ when $A\leq C B$ for some constant $C$ independent of $A$ and $B$.
 
 We visualize the region of boundedness as a convex polytope with  $2^m+m-1$ vertices contained in  
 the cube $[0,1]^m$  with coordinates $(\frac{1}{p_1},\dots,\frac{1}{p_m})$ (see Figure \ref{F1}). 
 The closure of this region, which we denote by $\mathcal H$, is 
 obtained as the intersection of $[0,1]^m$ with the half space $\sum_{j=1}^m \frac{1}{p_j} 
\le  \frac{mn-1}{n}$.  
Strong type boundedness at a point $  (\frac{1}{p_1},\dots,\frac{1}{p_m})$ means that (\ref{ST}) is satisfied; similarly for weak type and Lorentz space bounds.  To better describe this region, 
we define $v_j = (1,\dots, 1, \frac{n-1}{n}, 1, \dots, 1)$ for $j=1,\dots, m$ and $V=conv\{v_1,\dots,v_m\} $, the 
closed convex hull of the $v_j$'s.     We also denote by $\partial R$ the boundary of a region $R$ in $\mathbb R^m$.

\begin{theorem} \label{TM1} Let $n\geq 2$, $1\le  p_1,\dots ,p_m \le  \infty$ and $\sum_{j=1}^m \frac{1}{p_j}=\frac{1}{p}$.  Then the multilinear spherical maximal function $\M$ in (\ref{MS}) satisfies the following estimates: \\
\noindent {Case I:} If   $1<p_j<\infty$ for all $j\in \{1,\dots , m\}$, then $\M$
  is bounded from $L^{p_1}(\mathbb R^n)\times\dots\times L^{p_m}(\mathbb R^n)$ to $L^p(\mathbb R^n)$ if and only if $p>\frac{n}{mn-1}$.  \\
\noindent {Case II:} When   $\frac{1}{p_j} \in \{0,1\}$ for some $j$ and $p> \frac{n}{mn-1}$  we have: 
	\begin{enumerate}[label=\alph*) ]
		\item[(a)]	At the vertex $(0,\dots, 0)$   the strong type estimate (\ref{ST}) holds. 
		\item[(b)]   At the $2^m - 2$ vertices of $[0,1]^m$ except $(0,\dots, 0)$ and $(1,\dots, 1)$   the weak type estimate (\ref{WT}) holds.
	\end{enumerate}	
	Let $1\le k<m$. 
At each   open $k$-dimensional face  of  $\partial [0,1]^m \cap  \mathcal H$, described as the  set of all  points 
$(\frac{1}{p_1},\dots,\frac{1}{p_m})$ on the boundary of $[0,1]^m \cap  \mathcal H$ with exactly $m-k$  fixed coordinates in $\{0,1\}$, we  have:
     \begin{enumerate}[label=\alph*) ]
	 	\addtocounter{enumi}{2}
	 	\item[(c)] 	If all $m-k$ fixed coordinates are $0$, then  the strong type estimate (\ref{ST}) holds for all $n\geq 2$. 
	 	\item[(d)]  If at least one fixed coordinate equals $1$, then   the strong type estimate (\ref{ST}) holds when $n\geq 3$. 
	 \end{enumerate}
\noindent {Case III:} When $p=\frac{n}{mn-1} $  (critical exponent), then we have when $n\ge 3$: 
 	 \begin{enumerate}[label=\alph*) ]
 	 	\addtocounter{enumi}{4}
 	 	\item[(e)]  On the boundary of $V$  we have the Lorentz space estimate (\ref{RWT}).
 		\item[(f)]  On the interior of $V$ we have  the weak type estimate (\ref{WT}). More generally, we have  
 		\begin{equation}\label{LT}
		 \|\M(f_1,\dots,f_m)\|_{L^{\frac{n}{mn-1},\infty}}\lesssim \prod_{j=1}^m \|f_j\|_{L^{p_j,s_j}},
 		\end{equation}
 	for all $s_1,\dots, s_m>0$ such that $\sum\limits_{j=1}^m \frac{1}{s_j} = \frac{mn-1}{n}.$
 	\end{enumerate}

\end{theorem}

\begin{remarks*}  
{\bf 1.}   Using a well-known theorem of Stein and Str\"omberg \cite{SS1983}, we can see that in the case of the largest open set, and the endpoint estimates (a) and (c) above, the implicit constant can be taken to be independent of the dimension $n$.
		
\noindent {\bf 2.} As was noted in \cite{JL2019}, the method used in the proof of Theorem \ref{TM1} also 
yields  bounds for the stronger multi(sub)linear operator 
		\begin{equation} \label{SMF} \mathscr{M} (f_1,\dots,f_m) (x) := \sup_{t_1,\dots,t_m>0}\left| \int_{\mathbb S^{mn-1}}  \prod_{j=1}^{m} f_j(x-t_j y^j) d\s_{mn-1}(y^1,\dots,y^m)\right|
		\end{equation}
		in the same ranges of $p_j$'s as $\M$. Also, trivially, the counterexamples provided for the unboundedness of $\M$ also work for $\mathscr{M}$.

\noindent {\bf 3.}   When $n=1$ and $m\geq 3$, estimates in the case $L^{p}(\R)\times L^\infty(\R)\times\cdots\times L^\infty (\R) \to L^{p}(\R)$ for $p>\frac{m}{m-1}$ follow from the classical theorem of Rubio de Francia in \cite{R1986}. However the 
		optimal results in the case $n=1$   remain  open.
\end{remarks*}

As an example we graph the area of boundedness for the trilinear spherical maximal function.

\newcommand{\Depth}{3.5}
\newcommand{\Height}{3.5}
\newcommand{\Width}{3.5}
\begin{figure}[h]
\begin{center}
	\begin{tikzpicture}
	\coordinate (O) at (0,0,0);
	\coordinate (A) at (0,\Width,0);
	\coordinate (B) at (0,\Width,\Height);
	\coordinate (C) at (0,0,\Height);
	\coordinate (D) at (\Depth,0,0);
	\coordinate (E) at (\Depth,\Width,0);
	\coordinate (F) at (\Depth,\Width,\Height);
	\coordinate (G) at (\Depth,0,\Height);
	\coordinate (v1) at (\Depth,\Height,0.7*\Width);
	\coordinate (v3) at (\Depth,0.7*\Height,\Width);
	\coordinate (v2) at (0.7*\Depth,\Height,\Width);
	\coordinate (Z) at (0,1.3*\Width,0);
	\coordinate (Y) at (1.3*\Depth,0,0);
	\coordinate (X) at (0,0,1.3*\Height);
	\draw[white,fill=yellow!35] (O) -- (C) -- (G) -- (D) -- cycle;
	\draw[white,fill=yellow!20] (O) -- (A) -- (E) -- (D) -- cycle;
	\draw[white,fill=yellow!20] (O) -- (A) -- (B) -- (C) -- cycle;
	\draw[white,fill=blue!15,opacity=0.4] (D) -- (E) -- (v1) -- (v3) -- (G) -- cycle;
	\draw[white,fill=blue!15,opacity=0.4] (C) -- (B) -- (v2)-- (v3) -- (G) -- cycle;
	\draw[white,fill=blue!15,opacity=0.4] (A) -- (B) -- (v2)-- (v1)  -- (E) -- cycle;
	\draw[red, fill=red!30,opacity=0.8] (v1) -- (v2) -- (v3) -- cycle; 
	\draw[blue,densely dotted] (O) -- (C);
	\draw[blue,densely dotted] (O) -- (D);
	\draw[blue,densely dotted] (O) -- (A);
	\draw[blue] (B) -- (A);
	\draw[blue] (B) -- (C);
	\draw[blue] (C) -- (G);
	\draw[blue] (E) -- (A);
	\draw[blue] (E) -- (D);
	\draw[blue] (G) -- (D);
	\draw[blue] (G) -- (v3);
	\draw[blue] (E) -- (v1);
	\draw[blue] (B) -- (v2);
	
	\node [above=1mm of v2] {$v_2$};
	\node [right=2mm of v3] {$v_3$};

	\draw[white,fill=white,opacity=0.2] (v2) -- (F) -- (v3) -- cycle;
	\draw[white,fill=white,opacity=0.2] (v2) -- (F) -- (v1) -- cycle;
	\draw[white,fill=white,opacity=0.2] (v1) -- (F) -- (v3) -- cycle;
	\draw[red] (F) -- (v1);
	\draw[red] (F) -- (v2);
	\draw[red] (F) -- (v3);
	\node [right=2mm of v1]  {$v_1=(\frac{n-1}{n} , 1, 1)$};
	\node [above=1mm of v2] {$v_2$};
	\node [right=2mm of v3] {$v_3$};
	
	\foreach \rr in {v1,v2,v3} {\filldraw[red](\rr) circle(1.5pt);}
	\foreach \rr in {O, A, B, C, D, E, G} {\filldraw[blue](\rr) circle(1.2pt);}
	\node [above right =1mm of O]  {O};
	\node [above right =1mm of A]  {C};
	\node [above=1mm of B]  {F};
	\node [above=1mm of E]  {E};
	\node [below=1mm of G]  {D};
	\node [below=1mm of D]  {B};
	\node [below right=1mm of C]  {A};
	
	
	\node [above left =1mm of X]  {$1/p_1$};
	\node [above=1mm of Y] {$1/p_2$};
	\node [left=1mm of Z] {$1/p_3$};
	\draw[->,black,densely dotted] (C) -- (X);
	\draw[->,black,densely dotted] (D) -- (Y);
	\draw[->,black,densely dotted] (A) -- (Z);

	\end{tikzpicture}
\caption[Figure 1]{$L^{p_1}\times L^{p_2}\times L^{p_3}\to L^p$ boundedness of the trilinear spherical maximal operator ($n\geq 2$). }\label{F1}
\end{center}
\end{figure}
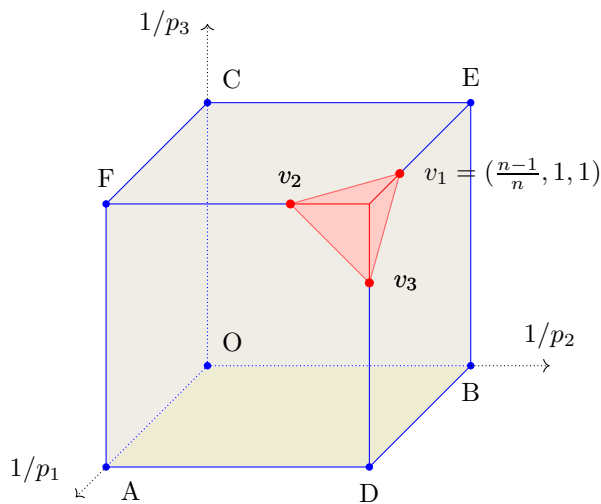
The counterexamples claimed in Theorem~\ref{TM1}  are contained in Proposition~\ref{P1}. This is obtained by  
an adaptation of the examples   in \cite[Proposition 7]{BGHHO2018}. 
\begin{prop} \label{P1} Let $1\leq p_1,\dots ,p_m \leq \infty$ and $\sum\limits_{j=1}^m \frac{1}{p_j}=\frac{1}{p}$. The multilinear spherical maximal function $\M$ is unbounded from $L^{p_1}\times\dots\times L^{p_m}$ to $L^p$ when $p\leq \frac{n}{mn-1}$ and $n\geq 2$.
\end{prop}

The following proposition provides a negative answer to a question posed  by Jeong and Lee in \cite{JL2019} regarding a strong type $L^1(\R^n)\times L^\infty(\R^n) \to L^1(\R^n)$ bound for the bilinear spherical maximal function. 

\begin{prop} \label{P2} Let $p_j \in \{1,\infty\}$ for all $j=1,\dots,m$. Then the strong type estimate $\|\M(f_1,\dots,f_m)\|_{L^p}$ $\lesssim \prod_{j=1}^m \|f_j\|_{L^{p_j}}$ holds if and only if $p_j=\infty$ for all $j=1,\dots,m$.
\end{prop}

The optimality of the  weak type and Lorentz space estimates of Theorem \ref{TM1} not covered in Proposition~\ref{P2} remains open.

\vspace{.1in}

\section*{Proof of Theorem \ref{TM1}}
The following decomposition lemma will be crucial in the sequel. This result can be found for instance in \cite[Appendix D.2]{GClassical}, \cite{BGHHO2018}, or \cite{HHY2019},  but we include a proof in the appendix for the sake of completeness. 
\begin{lemma} \label{L1} Let $1\leq k < m$ and $n\geq 2$. For a function $F(y^1,\dots, y^m)$ defined in $\R^{mn}$ with $y^j \in\R^n$, $j=1,\dots, m$ , we have 
	\begin{eqnarray}\label{CoV}  \int_{\mathbb S^{mn-1}} F(y^1,\dots, y^m) d\s_{mn-1}(y^1,\dots,y^m) =&\\
	 =\int_{B^{kn}} \int_{\ryk \mathbb S^{(m-k)n-1}}  F(y^1,\dots, y^m) &d\s^{\ryk}_{(m-k)n-1}(y^{k+1},\dots,y^m) \dfrac{dy^1\cdots dy^{k}}{\sqrt{1 - \sum_{j=1}^{k}|y^j|^2} }\nonumber,
	\end{eqnarray}
where $B^{kn} = B^{kn}(0,1)$ is the unit ball in $\R^{kn}$, $\ryk = \sqrt{1 - \sum_{j=1}^{k}|y^j|^2}$ and $d\s^{\ryk}_{(m-k)n-1}$ is the normalized surface measure on $\ryk \mathbb S^{(m-k)n-1}$.
\end{lemma}
 
\begin{proof}[Proof of Theorem \ref{TM1}]
 To avoid technicalities arising from interpolating sublinear operators, we consider the following linerization of the maximal operator. For a measurable function $\tau: \R^n \to [0,\infty)$ we define
 \begin{equation*}
 T(f_1,\dots,f_m) (x) := \int_{\mathbb S^{mn-1}}  \prod_{j=1}^{m} f_j(x-\tau(x) y^j) d\s_{mn-1}(y^1,\dots,y^m).
 \end{equation*}
Since the boundedness of $T$ on some spaces implies the boundedness for $\M$ on the same spaces, it is enough to show 
\begin{equation*}
 \|T(f_1,\dots,f_m)\|_{L^p}\leq C \prod_{j=1}^m \|f_j\|_{L^{p_j}},
\end{equation*}
for some $C$ independent of $\tau$. Since $|T(f_1,\dots,f_m)| \leq \M(f_1,\dots,f_m)$, applying Lemma~\ref{L1} with $k=m-1$ yields the following $m$ pointwise estimates: 
\begin{equation} \label{PW} T (f_1,\dots,f_m) (x) \lesssim \prod_{j\neq k} Mf_j(x)\cdot Sf_k(x),\quad k=1,\dots ,m,
\end{equation}
where $Mf_j$ is the Hardy-Littlewood maximal function of $f_j$ and $Sf_k$ is the linear spherical maximal function of $f_k$. For $M$ it is well known that $\|Mf\|_{L^{p}} \lesssim \|f\|_{L^p}$ for $p>1$ and $\|Mf\|_{L^{1}} \lesssim \|f\|_{L^{1,\infty}}$. Also, for $n\geq 2$, $\|Sf\|_{L^{p}} \lesssim \|f\|_{L^p}$ if and only if $p>\frac{n}{n-1}$. Therefore, by H\"older's inequality we obtain the following $m$ estimates  
\begin{eqnarray*} 
\|T(f_1,\dots,f_m)\|_{L^p} &\leq&  \prod_{j\neq k} \|Mf_j\|_{L^{p_j}}\cdot \|Sf_k\|_{L^{p_k}}, \quad k=1,\dots,m\\
&\lesssim& \prod_{j=1}^m \|f_j\|_{L^{p_j}},
\end{eqnarray*}
when $p_k>\frac{n}{n-1}$, $1<p_1,\dots ,p_m \leq \infty$ and $\sum\limits_{j=1}^m \frac{1}{p_j}=\frac{1}{p}$. Thus, applying (complex) interpolation between these estimates and the trivial $L^\infty\times\cdots\times L^\infty\to L^\infty$ bound, we obtain the boundedness in the largest possible open set of exponents, as well as the endpoint estimates $(a)$, $(b)$ and $(c)$ in the statement of the theorem (see \cite{GLLZ2012} and \cite[Theorem 7.2.2]{GModern} for the interpolation result we used for $(c)$). 

For the estimates in $(d)-(f)$, we will use Bourgain's restricted weak type endpoint bound for the linear spherical maximal function, which only holds for $n\geq 3$ (for $n=2$ Seeger, Tao, and Wright \cite{STW2003} showed that the restricted weak type inequality fails in the linear case). So for $n\geq 3$ we have the following $m$ estimates:
\begin{eqnarray*} \|T(f_1,\dots,f_m)\|_{L^{\frac{n}{mn-1},\infty}} \leq  \prod_{j\neq k} \|Mf_j\|_{L^{1,\infty}}\cdot \|Sf_k\|_{L^{\frac{n}{n-1},\infty}}\lesssim \prod_{j\neq k} \|f_j\|_{L^{1}}\|f_k\|_{L^{\frac{n}{n-1},1}}.
\end{eqnarray*}	
Interpolating these estimates with the estimates in $(c)$, we conclude $(d)$. 

Moreover, trivially $\|T(\chi_{F_1},\dots,\chi_{F_m})\|_{L^{\frac{n}{mn-1},\infty}} \lesssim \prod_{j=1}^m|F_j|^{\frac{1}{p_j}}$ for any measurable sets $F_1,\dots,F_m$ and for all  $1\leq p_1,\dots,p_m\leq \frac{n}{n-1}$ such that $\sum_{j=1}^m \frac{1}{p_j} = \frac{mn-1}{m}$. Since $L^{\frac{n}{mn-1},\infty}(\R^n)$ is $\frac{n}{mn-1}$ - convex, the estimates imply 
\begin{eqnarray*} \|T(f_1,\dots,f_m)\|_{L^{\frac{n}{mn-1},\infty}} \lesssim \prod_{j=1}^m \|f_j\|_{L^{p_j,\frac{n}{mn-1}}}.
\end{eqnarray*}	
Thus, using multilinear interpolation we conclude $(e)$ and $(f)$ (see \cite[Lemma 2.1 and Proposition 2.2]{BOS2009} for the multilinear interpolation result used here).\end{proof}

\section*{Counterexamples}

\begin{proof}[Proof of Proposition \ref{P1}] 
We consider the functions $f_j(y) = |y|^{-\frac{n}{p_j}} \log\left(\frac{1}{|y|}\right)^{-\frac{m}{p_j}}\chi{_{_{|y|\leq \nu_j}}}$, where $\nu_j= e^{-m/n}/100$ when $j\leq m-1$ and $\nu_m = e^{-m/n}/2$. Then $f_j\in L^{p_j}(\R^n)$. Since the mapping $(y^1,\dots,y^m)\mapsto (Ay^1,\dots,Ay^m)$, with $A\in SO_n$ is an isometry on $\mathbb S^{mn-1}$, we see that we can estimate $\M(f_1,\dots,f_m)$ from below by $M_{\sqrt{m}R}(f_1,\dots,f_m)(Re_1)$, for some $R$ large, where $e_1 = (1,0,\dots,0)\in \R^n$ and
\begin{equation*}
\M_s (f_1,\dots,f_m) (x) := \int_{\mathbb S^{mn-1}}  \prod_{j=1}^{m} f_j(x-s y^j) d\s_{mn-1}(y^1,\dots,y^m).
\end{equation*}
 Let $Y=(y^1,\dots, y^{m-1})\in \R^{(m-1)n}$, $y^m=:z =(z',z_n)$, with $z' =(y^m_1,\dots,y^{m}_{n-1})\in\R^{n-1}$ and $z_n = \sqrt{1 - \sum_{j=1}^{m-1} |y^j|^2 - |z'|^2}$. Moreover, we let $E=(e_1,\dots,e_1)\in \R^{(m-1)n}$. Applying Lemma~\ref{L1} with $k=m-1$, we have that 
\begin{align*} M_{\sqrt{m}R}(f_1,\dots,f_m)&(Re_1)= \\
& = \int_{B^{(m-1)n}} \prod_{j=1}^{m-1} f_j(Re_1 - \sqrt{m}R y^j) \int_{\ry \mathbb S^{n-1}} f_m(Re_1 - \sqrt{m}R z) d\s^{\ry}_{n-1}(z) \dfrac{dY}{\sqrt{1 - |Y|^2} }.\\
\end{align*}
First we focus on the inner integral, namely
\begin{align*} I =\int_{D} |Re_1 - \sqrt{m}Rz|^{-\frac{n}{p_m}}\left(-\log\, |Re_1 - \sqrt{m}Rz|\right)^{-\frac{m}{p_m}} d\s^{\ry}_{n-1}(z),
\end{align*}
where $D = \ry \mathbb S^{n-1} \cap B^n(\frac{e_1}{\sqrt{m}}, \frac{\nu_m}{\sqrt{m}R})$. For $z_0 \in \ry \mathbb S^{n-1} \cap \partial B^n(\frac{e_1}{\sqrt{m}}, \frac{\nu_m}{\sqrt{m}R})$, we let $\theta$ be the angle between the vectors $z_0$ and $e_1$, which is the largest one between $z\in D$ and $e_1$. Then $\theta$ is small since $R$ is large and $|D|\sim (\sqrt{1-|Y|^2}\theta)^{n-1} \sim \theta^{n-1}$. Using the fact that for $\theta$ small, $\theta^2\sim \sin^2\theta = 1-\cos^2\theta \sim 1-\cos\theta$, and the law of cosines
\[\frac{1}{4mR^2}=1-|Y|^2 + \frac{1}{m} - 2 \sqrt{1-|Y|^2}\frac{1}{\sqrt{m}}\cos\theta,\]
we obtain that $\theta^2 \sim \frac{1}{4mR^2} - \left(\sqrt{1-|Y|^2} - \frac{1}{\sqrt{m}}\right)^2$. In turn, since $\sqrt{1-|Y|^2}> \frac{1}{2}$ when $R>2(m-1)$, we have that $ \left| \sqrt{1-|Y|^2} - \frac{1}{\sqrt{m}}\right| \leq \frac{4(m-1)}{100\sqrt{m} R}$ 
and thus we conclude that $\theta \geq C/R$. The same calculation also yields that $\left| \sqrt{1-|Y|^2} - \frac{1}{\sqrt{m}}\right|\lesssim \left|\frac{E}{\sqrt{m}}- Y\right|$, that will be used later in the proof. Hence, we can bound $I$ from below by 
\[\int_0^\theta \int_{\ry \sin\alpha \mathbb S^{n-2}} |Re_1 - \sqrt{m}Rz|^{-\frac{n}{p_m}}\left(-\log\, |Re_1 - \sqrt{m}Rz|\right)^{-\frac{m}{p_m}} d\s^{\ry\sin\alpha}_{n-2} d\alpha,\] 
where $|z| = \ry = \sqrt{1-|Y|^2} \approx \frac{1}{\sqrt{m}}$ and $z=\ry\cos\alpha$. By symmetry, it suffices to consider the case $\ry<\frac{1}{\sqrt{m}}$. Let $\beta$ be the angle such that $|\frac{e_1}{\sqrt{m}} - z| = 2 |\ry - \frac{1}{\sqrt{m}}|$. Then $\beta\sim |\ry - \frac{1}{\sqrt{m}}|$. On the other hand, when $\alpha = 0$, we have $|\frac{e_1}{\sqrt{m}} - z| = |\ry - \frac{1}{\sqrt{m}}|$. Therefore, for all $\alpha\in [0,\beta]$ we have that 
\[|z-\frac{e_1}{\sqrt{m}}| \sim |\ry - \frac{1}{\sqrt{m}}| \lesssim |Y-\frac{E}{\sqrt{m}}, |\]
as was noted above. Thus, using Lemma~\ref{L2}, we obtain
\begin{align*} I&\geq C \int_0^\theta \int_{\ry \sin\alpha\mathbb S^{n-2}} \frac{|Re_1 - \sqrt{m}Rz|^{1-n}}{|Re_1 - \sqrt{m}Rz|^{\frac{n}{p_m}-n+1}\left(-\log\, |Re_1 - \sqrt{m}Rz|\right)^{\frac{m}{p_m}}} d\s_{n-2}^{\ry \sin\alpha} (z) d\alpha\\
&\geq CR^{1-n} |RE - \sqrt{m}RY|^{-\frac{n}{p_m}+n-1}\left(-\log\, |RE - \sqrt{m}RY|\right)^{-\frac{m}{p_m}} |\sqrt{m}\ry -1|^{1-n} \int_0^{C(1-\sqrt{m}\ry )} \sin^{n-2}\alpha\, d\alpha\\
&\geq CR^{1-n} |RE - \sqrt{m}RY|^{-\frac{n}{p_m}+n-1}\left(-\log\, |RE - \sqrt{m}RY|\right)^{-\frac{m}{p_m}}.
\end{align*}

Also, for any $0\leq j\leq m-1$, we have the trivial bound \(|Re_1 - \sqrt{m}Ry^j| \leq \left(\sum_{j=1}^{m-1} |Re_1 - \sqrt{m}Ry^j|^2 \right)^{1/2} = |RE - \sqrt{m}RY|.\) Thus using Lemma~\ref{L2} again, we see that
\begin{align*} &M_{\sqrt{m}R}(f_1,\dots,f_m)(Re_1) \\
&\geq C R^{1-n} \int_{B^{(m-1)n}(\frac{E}{\sqrt{m}},\frac{1}{50\sqrt{m}R}) } |RE - \sqrt{m}RY|^{-\frac{n}{p}+n-1}\left(-\log\, |RE - \sqrt{m}RY|\right)^{-\frac{m}{p}} dY\\
&\geq C R^{1-mn} \int_{B^{(m-1)n}(0,\frac{1}{50})} |w|^{-\frac{n}{p}+n-1}\left(-\log |w|\right)^{-\frac{m}{p}} dw\\
&\geq C R^{1-mn} \int_{0}^\frac{1}{50} r^{-\frac{n}{p}+mn-2}\left(-\log r\right)^{-\frac{m}{p}} dr\\
&=\begin{cases} C R^{1-mn} & \text{if } p = \frac{n}{mn-1}\\
\infty & \text{if } p < \frac{n}{mn-1}. \end{cases}
\end{align*}
We therefore conclude that $\M(f_1,\dots,f_m)$ is not in $L^p$ for any $ p < \frac{n}{mn-1}$ and when $p = \frac{n}{mn-1}$, $\M(f_1,\dots,f_m)(x) \gtrsim |x|^{1-mn}$ for all $|x|$ large enough and thus it is also not in $L^{\frac{n}{mn-1}}(\R^n)$ when $ p = \frac{n}{mn-1}$.
\end{proof}

\begin{proof}[Proof of Proposition \ref{P2}] The $L^\infty\times\cdots\times L^\infty \to L^\infty$ holds trivially and the multilinear maximal function is unbounded from $L^1\times\cdots\times L^1\to L^{\frac{1}{m}}$ by Proposition \ref{P1}. Let $1\leq k<m$. By symmetry, it is enough to show that the strong type estimate (\ref{ST}) fails at the point $(1,\dots,1,0,\dots,0)$. Let $f_1(y) = \dots = f_k(y) = \chi_{B^n(0,1/2)} (|y|)$ and $f_{k+1}(y) = \dots= f_m = 1$. Then, similar to the proof of Proposition~\ref{P1}, we obtain the following pointwise bound
	\begin{align*} \M(f_1,\dots,f_m) (x) & = \M(f_1,\dots,f_m) (Re_1) \geq \M_{\sqrt{m} R} (f_1,\dots,f_m) (Re_1)\\
	& = \int_{\mathbb S^{mn-1}} \prod_{j=1}^k \chi_{B^n(0,1/2)} (|Re_1 - \sqrt{m}R y^j|) d\s_{mn-1}(y^1,\dots,y^m)\\
	&\geq \int_{\mathbb S^{mn-1}} \chi_{B^{kn}(0,1/2)} (|R E_k - \sqrt{m}R Y_k|) d\s_{mn-1}(y^1,\dots,y^m),
	\end{align*}
where $E_k = (e_1,\dots,e_1)$ and $Y_k = (y^1,\dots,y^k)$ are vectors in $\R^{kn}$. Then, applying Lemma~\ref{L1}, we have that
\begin{align*} &\int_{\mathbb S^{mn-1}}  \chi_{B^{kn}(0,1/2)} (|R E_k - \sqrt{m}R Y_k|) d\s_{mn-1}(y^1,\dots,y^m)\\
&= \int_{B^{kn}}  \chi_{B^{kn}(0,1/2)} (|R E_k - \sqrt{m}R Y_k|) \int_{\ryk \mathbb S^{(m-k)n-1}}  d\s^{\ry}_{[(m-k)n-1]}(y^{k+1},\dots,y^m) \dfrac{dY_k}{\sqrt{1 - |Y_k|^2} }\\
&\gtrsim \int_{B^{kn}}  \chi_{B^{kn}(0,1/2)} (|R E_k - \sqrt{m}R Y_k|)dY_k,\\
\end{align*}
since $1 - |Y_k|^2 \geq \frac{1}{2m}$ when $R$ is large enough. Therefore
		\begin{align*} \M(f_1,\dots,f_m) (x) & = \M(f_1,\dots,f_m) (Re_1) \geq \M_{\sqrt{m} R}(f_1,\dots,f_m) (Re_1) \gtrsim R^{-kn},	\end{align*}
and therefore $\M$ does not map $L^1\times\cdots\times L^1 \times L^\infty \times \cdots \times L^\infty\to L^{\frac{1}{k}}$.
\end{proof}

\section*{Appendix}

\begin{proof}[Proof of  Lemma~\ref{L1}] For $y^j=(y^j_1,\dots,y^j_n)\in\R^n$, $j=1,\dots,m$, we set $Y=(y^1,\dots, y^{k})\in \R^{kn}$, $(y^{k+1},\dots, y^m)=:Z =(Z',z_n)$, with $Z' = (y^{k+1},\dots,y^{m-1},y^m_1,\dots, y^{m}_{n-1})\in\R^{(m-k)n-1}$ and $z_n = \sqrt{1 - \sum_{j=1}^{k} |y^j|^2 - |Z'|^2}$. For the sake of clarity in notation, we write $\ry = \sqrt{1-\sum_{j=1}^k |y^j|^2}$, instead of $\ryk$.
	
	Setting $Z/\ry=W=(W',w_n)$, we express the right hand side of (\ref{CoV}) as
	\begin{align*} &\int_{B^{kn}} \int_{\ry \mathbb S^{(m-k)n-1}} F(Y, Z) d\s^{\ry}_{[(m-k)n-1]}(Z) \dfrac{dY}{\sqrt{1 - |Y|^2} } \\
	=&\int_{B^{kn}} \ry^{(m-k)n-1} \int_{\mathbb S^{(m-k)n-1}} F(Y, \ry W) d\s_{[(m-k)n-1]}(W) \dfrac{dY}{\sqrt{1 - |Y|^2} }\\
	=&\int_{B^{kn}} \ry^{(m-k)n-1} \int_{B^{(m-k)n-1}} [F(Y, \ry W',\ry w_n)+F(Y, \ry W',-\ry w_n)] \frac{dW'}{\sqrt{1-|W'|^2} }\dfrac{dY}{\sqrt{1 - |Y|^2} }\\
	=&\int_{B^{kn}} \int_{\ry B^{(m-k)n-1}} [F(Y, Z',z_n)+F(Y, Z',-z_n)] \frac{dZ'}{\sqrt{1-|W'|^2} }\dfrac{dY}{\sqrt{1 - |Y|^2} }\\
	=&\int_{B^{kn}} \int_{\ry B^{(m-k)n-1}} [F(Y, Z',z_n)+F(Y, Z',-z_n)] \dfrac{dZ'dY}{\sqrt{1 - |Y|^2 - |Z'|^2} },\\
	\end{align*}
	as one can verify that $\sqrt{1-|W'|^2}\sqrt{1-|Y|^2} = \sqrt{1 - |Y|^2 - |Z'|^2}$. Using that $B^{mn-1}$ is equal to the disjoint union of the sets $\{(Y,\ry v)\,:\, v\in B^{(m-k)n-1}\}$, we see that the last integral is equal to 
	\begin{equation*} \int_{B^{mn-1}}[ F(Y,Z',z_n) + F(Y,Z',-z_n)] \frac{dYdZ'}{\sqrt{1-|Y|^2 - |Z'|^2}},
	\end{equation*}
	which, in turn, is equal to $\int_{\mathbb S^{mn-1}} F(Y,Z) d\s_{mn-1}(Y,Z)$ (see \cite[Appendix D.5]{GClassical}).
\end{proof}
\begin{lemma} \label{L2} Let $r_1,r_2>0$, $t,s< e^{-\frac{r_2}{r_1}}$ and $t\leq C s$ for some $C\geq 1$. Then there exists an absolute constant $C'$ (depending only on $C,$ $r_1,$ $r_2$) such that
	\begin{equation}\label{l2} s^{-r_1}\left(\log \frac{1}{s}\right)^{-r_2}\leq C' t^{-r_1}\left(\log \frac{1}{t}\right)^{-r_2}.
	\end{equation}
	\begin{proof} Define $F(x) = x^{r_1}\left(\log x\right)^{-r_2}$. Differentiating $F$, we see that $F$ is increasing when $x> e^{\frac{r_2}{r_1}}$ and so for $s< e^{-\frac{r_2}{r_1}}$,
		\begin{equation*} 
		F\left(\frac{1}{s}\right) =  s^{-r_1}\left(\log \frac{1}{s}\right)^{-r_2}\leq C^{r_1}(Cs)^{-r_1}\left(\log \frac{1}{Cs}\right)^{-r_2}= C^{r_1} F\left(\frac{1}{Cs}\right) \leq C' F\left(\frac{1}{t}\right) = C'  t^{-r_1}\left(\log \frac{1}{t}\right)^{-r_2}.
		\end{equation*}
	\end{proof}
\end{lemma}
\noindent \textbf{Acknowledgment.} 
I would like to express my sincere gratitude to Professor L. Grafakos for his invaluable support.

\bibliographystyle{amsplain}

\Addresses
\end{document}